\documentclass{article}

\usepackage{amsmath,amsthm,amsfonts,amssymb}    

   \newcommand{\Z}{\mathbb{Z}}  		
   \newcommand{\Q}{\mathbb{Q}}  		
   \newcommand{\C}{\mathbb{C}}  		

   \theoremstyle{plain} 				
   \newtheorem{theorem}{Theorem}
   
   \newtheorem{lemma}{Lemma}

   \theoremstyle{definition}
   
   \newtheorem{definition}{Definition}
   \newtheorem{example}{Example}
   
   \theoremstyle{remark}
   \newtheorem*{remark}{Remark}

\begin{document}

\title{Hecke-Symmetry and Rational Period Functions}


\author{Dedicated to the memory of my teacher Marvin Knopp,\\ 
	who corrected many of my mistakes.\\[14pt]
	Wendell Ressler
	  Department of Mathematics \\
	  Franklin \& Marshall College \\
	  Lancaster, PA 17604-3003}

\maketitle

\begin{abstract}
In this paper
we continue work in the direction
of a characterization of 
rational period functions on the Hecke groups.
We examine the role that
Hecke-symmetry of poles
plays in this setting,
and pay particular attention
to non-symmetric irreducible systems of poles
for a rational period function.
This gives us a new expression
for a class of
rational period functions
of any positive even integer weight
on the Hecke groups.
We illustrate these properties 
with examples 
of specific rational period functions.
We also correct the wording of a theorem from 
an earlier paper.
\end{abstract}

\section{Introduction}
\label{sec:intro}
Marvin Knopp first 
defined and studied rational period functions (RPFs)
for automorphic integrals 
\cite{MR0344454}
and gave the first example
of an RPF
with nonzero poles 
\cite{MR0485700}.

Knopp \cite{MR623004}, 
Hawkins\footnote{``On rational period functions for the modular group,'' unpublished manuscript.} 
and Choie and Parson 
\cite{MR1045397,MR1103673} 
took the main steps toward an explicit characterization of RPFs on 
the modular group \(\Gamma(1)\).
Ash \cite{MR980298} 
gave an abstract
characterization, 
and then
Choie and Zagier \cite{MR1210514} and
Parson \cite{MR1210515} provided 
a more explicit characterization 
of the RPFs on \(\Gamma(1)\).
The explicit characterization uses continued fractions 
to establish a connection between the poles of 
RPFs and binary quadratic forms.

Schmidt \cite{MR1378572} generalized Ash's work, 
giving an abstract characterization of RPFs 
on any finitely generated Fuchsian group of the 
first kind with parabolic elements, 
a class of groups which includes the Hecke groups.
Schmidt \cite{MR1219337} 
and Schmidt and Sheingorn \cite{MR1362251} 
took steps 
toward an explicit characterization of RPFs 
on the Hecke groups 
using generalizations of  
classical continued fractions and binary quadratic forms.
This author continued that work 
in \cite{MR1876701} and \cite{MR2582982}.

The nonzero poles of RPFs on \( \Gamma(1) \)
have unique algebraic conjugates
because they are quadratic irrational numbers.
The nonzero poles of RPFs on a Hecke group are 
algebraic numbers,
usually
of degree greater than \( 2 \).
In \cite{MR1876701} and \cite{MR2582982}
we define the unique
Hecke-conjugate for any 
nonzero pole of an RPF on a Hecke group.
In this paper
we explore the roles that
Hecke-conjugation
and the related Hecke-symmetry of pole sets
play in constructing 
RPFs on the Hecke groups.

\section{Background}
\label{sec:background}

In this section
we list definitions and results 
necessary for working with 
Hecke-symmetry and rational period functions.
More details can be found in
\cite{MR1876701} and \cite{MR2582982}.

\subsection{Hecke groups and Hecke-symmetry}
\label{subsec:Hecke_groups}

Throughout this paper we
fix an integer \( p \geq 3 \)
and \( \lambda = \lambda_{p} = 2\cos(\pi/p) \).
Put \(S = S_{\lambda} 
   = \bigl( \begin{smallmatrix} 1&\lambda\\ 0&1 \end{smallmatrix} \bigr) \),
\(T = \bigl( \begin{smallmatrix} 0&-1\\ 1&0 \end{smallmatrix} \bigr)\),
and \(I = \bigl( \begin{smallmatrix} 1&0\\ 0&1 \end{smallmatrix} \bigr)\).
We define \emph{Hecke group}
\begin{equation*}
G_{p} = G(\lambda_{p}) = \langle S,T \rangle / \{\pm I\}.
\end{equation*}
If we put 
\( U = U_{\lambda} = S_{\lambda}T 
	= \bigl( \begin{smallmatrix} \lambda & -1\\ 1 & 0 \end{smallmatrix} \bigr) \)
we can write the group relations of \( G_{p} \) as
\( T^{2}=U^{p}=I\).
It is well-known that 
\(G_{p} = \mathrm{PSL}(2,\Z[\lambda_{3}]) = \Gamma(1)\);
however for the other Hecke groups 
\(G_{p} \subsetneqq \mathrm{PSL}(2,\Z[\lambda_{p}])\).
The entries of elements of \( G_{p} \)  are in 
\( \Z [\lambda_{p}] \), 
the ring of algebraic 
integers for \( \Q(\lambda_{p}) \).

Members of \( G_{p} \)
act on the Riemann sphere 
as M\"{o}bius transformations.
An element 
\(M = \bigl( \begin{smallmatrix} a&b\\ c&d \end{smallmatrix} \bigr) \in G_{p}\) 
is 
\emph{hyperbolic} if \( \vert a+d \vert >2 \),
\emph{parabolic} if \( \vert a+d \vert =2 \),
and \emph{elliptic} if \( \vert a+d \vert <2 \). 
We designate fixed points accordingly.
Hyperbolic M\"{o}bius transformations each have 
two distinct real fixed points.

The \emph{stabilizer} in \( G_{p} \) of a complex number \( z \),
\( \mathrm{stab}(z) = \{M \in G_{p} : Mz=z\} \),
is a cyclic subgroup of \( G_{p} \).
We define the
\emph{Hecke-conjugate} of any hyperbolic fixed point of
\( G_{p} \)
to be the other fixed point of the elements in its stabilizer.
We denote the Hecke-conjugate of 
\( \alpha \)
by
\( \alpha^{\prime} \).
If \( R \) is a set of hyperbolic fixed points of 
\( G_{p} \) we write
\( R^{\prime} = \{x^{\prime} : x \in R\} \).
We say that a set \( R \) has 
\emph{Hecke-symmetry}
if \( R = R^{\prime} \).
We note that
\( R \cup R^{\prime} \)
has Hecke-symmetry
for any set of hyperbolic points \( R \).

\subsection{\( \lambda \)-binary quadratic forms}
\label{subsec:BQFss}

Hawkins\footnote{``On rational period functions for the modular group,'' unpublished manuscript.} 
pointed out a deep connection between 
rational period functions 
for the modular group and classical
binary quadratic forms.
We exploit a similar connection
between 
rational period functions 
for Hecke groups
and generalized binary quadratic forms.

We let 
\( \mathcal{Q}_{p,D} \) denote the set
of binary quadratic forms
\[ Q(x,y) = Ax^{2}+Bxy+Cy^{2}, \]
with coefficients in  
\( \Z[\lambda_{p}] \),
and discriminant \( D \).
We also denote a form by 
\( Q=[A,B,C] \) and refer to it as a
\( \lambda \)-BQF.
We restrict our attention to indefinite forms.

We define an action
of \( G_{p} \)
on \( \mathcal{Q}_{p,D} \) by
\( \left(Q \circ M \right)(x,y) = Q(a x + b y, c x + d y) \)
for 
\( Q \in \mathcal{Q}_{p,D} \) 
and
\( M = \bigl( \begin{smallmatrix} a &b \\ c & d 
	      \end{smallmatrix} \bigr) \in G_{p} \).
This action preserves the discriminant
and partitions \( \mathcal{Q}_{p,D} \) into equivalence 
classes of forms.

In \cite{MR2582982}
we describe a one-to-one correspondence between hyperbolic fixed points
of \( G_{p} \) and certain 
\( \lambda \)-BQFs.
We first use a variant of Rosen's 
\( \lambda \)-continued fractions 
\cite{MR0065632,MR1219337}
to map every hyperbolic point 
\( \alpha \) to the unique primitive hyperbolic element
\( M_{\alpha} \in G_{p} \) with positive trace
that has \( \alpha \) as an attracting fixed point.
This mapping associates Hecke-conjugates with inverse elements in the Hecke group, 
that is,
\( M_{\alpha^{\prime}} = M_{\alpha}^{-1} \).
We also map every hyperbolic element
\( M = \left( \begin{smallmatrix} a&b\\ c&d \end{smallmatrix} \right) 
\in G_{p} \) 
with positive trace to 
a unique indefinite \( \lambda \)-BQF
\( Q \) such that 
the roots of
\( Q(z,1) = cz^{2} + (d-a)z - b \)
are the fixed points
of \( M \).
Every domain element for this map
is hyperbolic, 
so we call the images \emph{hyperbolic} \( \lambda \)-BQFs.

The composition of the two maps described above 
associates every hyperbolic fixed point
\( \alpha \) with a unique hyperbolic
\( \lambda \)-BQF 
\( Q_{\alpha} \).
These mappings are both injective,
and the inverse of the composition 
associates every 
hyperbolic
quadratic form
\( Q = [A,B,C] \) with the hyperbolic number
\( \alpha_{Q} = \frac{-B+\sqrt{D}}{2A} \),
where 
\( D \) is the discriminant of \( Q \).

Every equivalence class of \( \lambda \)-BQFs contains 
either all hyperbolic forms or no hyperbolic forms,
so we may label equivalence classes themselves as hyperbolic or 
non-hyperbolic.

If \( \mathcal{A} \) denotes an equivalence 
class of \( \lambda \)-BQFs, we put
\( -\mathcal{A} 
     = \left\{-Q \vert Q \in \mathcal{A}\right\} \).
Then \( -\mathcal{A} \) is another equivalence 
class of forms, 
not necessarily distinct from \( \mathcal{A} \).
If \( \mathcal{A} \) is hyperbolic, so is \( -\mathcal{A} \), 
and 
the numbers associated with the forms in 
\( -\mathcal{A} \) 
are the Hecke-conjugates of the numbers associated with the forms in 
\( \mathcal{A} \).

We call a hyperbolic \( \lambda \)-BQF \( Q = [A,B,C] \) 
\( G_{p} \)-\emph{simple} if
\( A > 0 > C \).
If \( Q \) is a simple 
\( \lambda \)-BQF, we say that 
the associated hyperbolic number
\( \alpha_{Q} \) is a \( G_{p} \)-\emph{simple} number.
A hyperbolic number
\( \alpha \) is simple 
if and only if 
\( \alpha^{\prime}<0<\alpha \).

If \( \mathcal{A} \) is a hyperbolic equivalence class of 
\( \lambda \)-BQFs we write
\( Z_{\mathcal{A}} 
    = \left\{x : Q_{x} \in \mathcal{A}, Q_{x} \textrm{ simple} \right\} \).
These sets are nonempty;
every hyperbolic equivalence class 
of \( \lambda \)-BQFs contains at 
least one simple form.

\subsection{Rational period functions}
\label{subsec:RPFs}

For 
\( M = \left( \begin{smallmatrix} * & * \\ c & d \end{smallmatrix} \right) \in G_{p} \)
and \( f(z) \) a function of a complex variable,
we define the
\emph{weight \( 2k \) slash operator}
\( f \mid_{2k} M = f \mid M \) to be
\begin{equation*}
\left( f \mid M \right)(z) = (cz+d)^{-2k}f(Mz).
\end{equation*}
\begin{definition}
Fix \( p \geq 3 \) and let \( k \in \Z^{+} \).
A 
\emph{rational period function (RPF) of weight \( 2k \) for \( G_{p} \)} 
is a rational function that satisfies the relations
\begin{equation}
\label{eq:first_relation}
q + q \mid T = 0,
\end{equation}
and
\begin{equation}
\label{eq:second_relation}
q + q \mid U + \cdots + q \mid U^{p-1} = 0.
\end{equation}
\end{definition}
This definition
is equivalent to Marvin Knopp's original definition
of rational period functions
for automorphic integrals
\cite{MR0344454}.

For any rational period function 
of weight \( 2k \) for \( G_{p} \)
we let \( P(q) \) denote the set of poles of \( q \).
Hawkins 
defined an
\emph{irreducible system of poles (ISP)},
which is the minimal set of poles
forced to occur together
by the relations 
\eqref{eq:first_relation} and \eqref{eq:second_relation}.

The poles of an RPF on \( G_{p} \) are all real,
and the nonzero poles are all 
hyperbolic fixed points of \( G_{p} \).
The set of positive poles 
in an ISP
is \( \mathcal{Z}_{\mathcal{A}} \)
for some hyperbolic equivalence class \( \mathcal{A} \).
The ISP associated with \( \mathcal{A} \)
is
\begin{align*}
P_{\mathcal{A}}
& = \mathcal{Z}_{\mathcal{A}} \cup T\mathcal{Z}_{\mathcal{A}} \\
& = \mathcal{Z}_{\mathcal{A}} \cup \mathcal{Z}_{-\mathcal{A}}^{\prime}.
\end{align*}

If \( q \) is an RPF of weight \( 2k \) on \( G_{p} \) 
with a pole \emph{only} at zero,
then 
\( q \) must have the form
\cite{MR774398}
\begin{equation}
  q_{k,0}(z) = 
    \begin{cases}
      a_{0}(1-z^{-2k}), & \text{if } 2k \neq 2, \\
      a_{0}(1-z^{-2}) + b_{1}z^{-1}, & \text{if } 2k = 2.
    \end{cases}
    \label{eq:RPFPoleatZero}
\end{equation}
If \( \alpha \neq 0 \)
occurs as a pole of an RPF \( q \)
of weight \( 2k \in 2\Z^{+} \)
on \( G_{p} \),
we define
\begin{equation}
q_{k,\alpha}(z)
= PP_{\alpha} 
  \left[ \frac{D^{k/2}}{Q_{\alpha}(z,1)^{k}} \right]
= PP_{\alpha} 
         \left[ \frac{(\alpha-\alpha^{\prime})^{k}}                                       
	 {(z-\alpha)^{k}(z-\alpha^{\prime})^{k}}
                 \right],
    \label{eq:PPatalpha}
\end{equation}
where \(D\) is the discriminant of 
the \( \lambda \)-BQF \( Q_{\alpha} \).
With this we have
the following expression,
valid for any RPF on \( G_{p} \)
\cite[p. 292]{MR1876701}.
\begin{theorem}
\label{thm:AnyRPFdiscription}
    An RPF of weight \( 2k \in 2\Z^{+} \) on \( G_{p} \) is of the form
    \begin{equation}
	\label{eq:RPFform}
        q(z) = \sum_{\ell=1}^{L} C_{\ell} 
    	    \left(\sum_{\alpha \in Z_{\mathcal{A}_{\ell}}}
	            q_{k,\alpha}(z) 
		- \sum_{\alpha \in Z_{-\mathcal{A}_{\ell}}}
		    q_{k,\alpha^{\prime}}(z)
	    \right) 
	    + c_{0}q_{k,0}(z)
	    + \sum_{n=1}^{2k-1}\frac{c_{n}}{z^{n}},
    \end{equation}
    where each
    \( \mathcal{A}_{\ell} \) is a \( G_{p} \)-equivalence class of 
        \( \lambda \)-BQFs,
    \( Z_{\mathcal{A}_{\ell}} \) is the cycle of positive poles 
        associated with \( \mathcal{A}_{\ell} \),
    \( q_{k,\alpha} \) is given by \eqref{eq:PPatalpha},
    \( q_{k,0} \) is given by \eqref{eq:RPFPoleatZero},
    and the \( C_{\ell} \) and \( c_{n} \) are all constants.
\end{theorem}

\section{RPFs with Hecke-symmetric pole sets}
\label{sec:RPFs_symmetric}

An RPF must have a particularly simple form
if every irreducible system of poles
has Hecke-symmetry,
and 
an RPF may have a simple form
if its complete set of poles
has Hecke-symmetry.

\subsection{A correction}

In Theorem 2 of \cite{MR1876701}
we consider the case of
pole sets
that have Hecke-symmetry,
with \( k \) odd.
The converse in Theorem 2
of is incorrect as stated.
The mis-statement 
involves the distinction between
\begin{enumerate}
\renewcommand{\labelenumi}{(\Alph{enumi})}
\item 
an RPF
for which every nonzero irreducible system of poles 
has Hecke-symmetry, and
\item 
an RPF
for which the complete set of nonzero poles 
has Hecke-symmetry.
\end{enumerate}
An RPF that satisfies (A)
must satisfy (B),
but the converse does not hold.
We state a corrected theorem 
that accounts for this distinction.

\begin{theorem}[corrected]
\label{thm.SymmetricISPcorrected}
    Suppose that \( k \) is an odd positive integer.
    
    Suppose \( q \) is an RPF of weight \( 2k \) on \( G_{p} \)
    with Hecke-symmetric
    irreducible systems of poles.
    Then
    \( q \) is of the form
    \begin{equation}
        q(z) = \sum_{\ell=1}^{L} d_{\ell}
               \sum_{\alpha \in Z_{\mathcal{A}_{\ell}}} 
	       Q_{\alpha}(z,1)^{-k}
	     + c_{0}q_{k,0}(z),
        \label{eq:RPFsymmetrickodd}
    \end{equation}
    where each
    \( \mathcal{A}_{\ell} \) 
    is a \( G_{p} \)-equivalence class of 
    \( \lambda \)-BQFs
    satisfying \( -\mathcal{A}_{\ell} = \mathcal{A}_{\ell} \),
    the \( d_{\ell} \) \( (1 \leq \ell \leq M) \) are constants,
    and \( q_{k,0}(z) \) is given by 
    \eqref{eq:RPFPoleatZero}.
   
    Conversely,
    any rational function of the form
    \eqref{eq:RPFsymmetrickodd}
    is an RPF of weight \( 2k \) 
    on \( G_{p} \) with 
    a set of nonzero poles that
    (taken in its entirety) 
    is
    Hecke-symmetric.
\end{theorem}

\begin{proof}
   The proof in \cite{MR1876701} is correct
   for this statement of the theorem.
\end{proof}

The corrected converse statement 
leaves open the possibility that
there exist
rational period functions
with Hecke-symmetric pole sets 
that do not have 
Hecke-symmetric
irreducible systems of poles.
Indeed, 
such rational period functions exist,
as the next example shows.

\begin{example}
\label{ex:RPF_nonsym_ISP}
Put \( \lambda = \lambda_{4} = \sqrt{2} \).
The smallest \( \lambda_{4} \)-BQF discriminant
with class number greater than \( 1 \)
is \( D = 14 \),
which has class number \( h_{4,14} = 2 \)
\cite{MR3078226}.
The two equivalence classes satisfy
\( \mathcal{A} \neq -\mathcal{A} \),
so we have
\( \mathcal{A}_{1} = - \mathcal{A}_{2} \).
We denote by \( \mathcal{A}_{1} \)
the equivalence class 
containing the simple BQFs
\begin{equation*}
	Q_{1,1} = [1,-\sqrt{2},-3] 
	\mbox{ and }
	Q_{1,2} = [1,\sqrt{2},-3],
\end{equation*}
and by \( \mathcal{A}_{2} \)
the equivalence class 
containing the simple BQFs
\begin{equation*}
	Q_{2,1} = [3,-\sqrt{2},-1] 
	\mbox{ and }
	Q_{2,2} = [3,\sqrt{2},-1].
\end{equation*}
The corresponding positive poles 
are
\begin{equation*}
\mathcal{Z}_{\mathcal{A}_{1}}
= \left\{ \alpha_{1}, \alpha_{2} \right\}
= \left\{ \frac{\sqrt{2}+\sqrt{14}}{2}, \frac{-\sqrt{2}+\sqrt{14}}{2} \right\},
\end{equation*}
and
\begin{equation*}
\mathcal{Z}_{\mathcal{A}_{2}}
= \left\{ \beta_{1}, \beta_{2} \right\}
= \left\{ \frac{\sqrt{2}+\sqrt{14}}{6}, \frac{-\sqrt{2}+\sqrt{14}}{6} \right\}.
\end{equation*}
The corresponding negative poles are
\begin{equation*}
T\mathcal{Z}_{\mathcal{A}_{1}}
= \left\{ \frac{\sqrt{2}-\sqrt{14}}{6}, \frac{-\sqrt{2}-\sqrt{14}}{6} \right\}
= \left\{ \beta_{1}^{\prime}, \beta_{2}^{\prime} \right\},
\end{equation*}
and
\begin{equation*}
T\mathcal{Z}_{\mathcal{A}_{2}}
= \left\{ \frac{\sqrt{2}-\sqrt{14}}{2}, \frac{-\sqrt{2}-\sqrt{14}}{2} \right\}
= \left\{ \alpha_{1}^{\prime}, \alpha_{2}^{\prime} \right\}.
\end{equation*}
If we put \( L=2 \), 
\( d_{1} = d_{2} = 1 \),
and \( c_{0} = 0 \)
in \eqref{eq:RPFsymmetrickodd}
we have the rational period function
\begin{align*}
q(z)
& = Q_{\alpha_{1}}(z,1)^{-k} + Q_{\alpha_{2}}(z,1)^{-k} 
	+ Q_{\beta_{1}}(z,1)^{-k} + Q_{\beta_{2}}(z,1)^{-k} \\
& = \frac{1}{(z^{2}-\sqrt{2}z-3)^{k}} + \frac{1}{(z^{2}+\sqrt{2}z-3)^{k}}
	+ \frac{1}{(3z^{2}-\sqrt{2}z-1)^{k}} + \frac{1}{(3z^{2}+\sqrt{2}z-1)^{k}}  \\
& = \frac{1}{(z-\alpha_{1})^{k}(z-\alpha_{1}^{\prime})^{k}}
	+ \frac{1}{(z-\alpha_{2})^{k}(z-\alpha_{2}^{\prime})^{k}}
	+ \frac{1}{3^{k}(z-\beta_{1})^{k}(z-\beta_{1}^{\prime})^{k}}
	+ \frac{1}{3^{k}(z-\beta_{2})^{k}(z-\beta_{2}^{\prime})^{k}}.
\end{align*}
The set of poles  
\begin{equation*}
P(q) = \left\{ \alpha_{1}, \alpha_{1}^{\prime},\alpha_{2}, \alpha_{2}^{\prime}, 
	\beta_{1}, \beta_{1}^{\prime}, \beta_{2}, \beta_{2}^{\prime} \right\},
\end{equation*}
is Hecke-symmetric,
while both irreducible systems of poles
\begin{equation*}
P_{\mathcal{A}_{1}}
	= \mathcal{Z}_{\mathcal{A}_{1}} \cup T\mathcal{Z}_{\mathcal{A}_{1}}
	= \left\{ \alpha_{1}, \alpha_{2}, 
		\beta_{1}^{\prime}, \beta_{2}^{\prime} \right\},
\end{equation*}
and 
\begin{equation*}
P_{\mathcal{A}_{2}}
	= \mathcal{Z}_{\mathcal{A}_{2}} \cup T\mathcal{Z}_{\mathcal{A}_{2}}
	= \left\{ \beta_{1}, \beta_{2}, \alpha_{1}^{\prime}, 
		\alpha_{2}^{\prime} \right\},
\end{equation*}
are not Hecke-symmetric.
\end{example}

\subsection{Pairs of non-symmetric ISPs}
\label{subsec:nonsymmetric_pairs_ISPs}

An important feature of the 
rational period function in Example 1 is
that the union of the two non-symmetric 
irreducible systems of poles
is Hecke-symmetric.
The next lemma shows that this always happens
for RPFs with Hecke-symmetric pole sets,
that is,
the non-symmetric ISPs always occur in pairs.

\begin{lemma}
Let \( q \) be an RPF 
of weight \( 2k \) on \( G_{p} \)
for \( k \in \Z^{+} \) and \( p \geq 3 \).
Suppose that \( q \)
has a Hecke-symmetric
set of nonzero poles.
If \( q \)
has any ISPs that are not Hecke-symmetric,
then such ISPs must occur in pairs.
That is,
if 
\( P_{\mathcal{A}} \)
is a non-Hecke-symmetric ISP for \( q \)
then 
\( P_{-\mathcal{A}} \)
is also a non-Hecke-symmetric ISP for \( q \)
that is
distinct from \( P_{\mathcal{A}} \),
and \( P_{\mathcal{A}} \cup P_{-\mathcal{A}} \)
is Hecke-symmetric.
\end{lemma}

\begin{proof}
If \( P_{\mathcal{A}} \)
is a non-Hecke-symmetric ISP 
then \( \mathcal{A} \neq -\mathcal{A} \)
and \( P_{\mathcal{A}} \neq P_{-\mathcal{A}} \).
Moreover,
\begin{align*}
P_{\mathcal{A}}^{\prime}
& = \left( \mathcal{Z}_{\mathcal{A}} \cup \mathcal{Z}_{-\mathcal{A}}^{\prime} \right)^{\prime} \\
& = \mathcal{Z}_{-\mathcal{A}} \cup \mathcal{Z}_{\mathcal{A}}^{\prime}  \\
& = P_{-\mathcal{A}},
\end{align*}
so 
\( P_{-\mathcal{A}} \neq P_{-\mathcal{A}}^{\prime} \),
and 
\( P_{-\mathcal{A}} \)
is also non-Hecke-symmetric.
On the other hand,
\begin{align*}
P_{\mathcal{A}} \cup P_{-\mathcal{A}}
& = \left( \mathcal{Z}_{\mathcal{A}} \cup \mathcal{Z}_{-\mathcal{A}}^{\prime} \right)
	\cup
	\left(\mathcal{Z}_{-\mathcal{A}} \cup \mathcal{Z}_{\mathcal{A}}^{\prime}  \right),
\end{align*}
which is Hecke-symmetric.
\end{proof}

\subsection{RPFs of any weight}
\label{subsec:RPFs_all_k}

The class of functions given by 
\eqref{eq:RPFsymmetrickodd}
contains rational period functions 
for odd values of \( k \)
(half of the weight).
In the next theorem 
we use sums of powers
of \( \lambda \)-BQFs
to write 
a class of functions that contains RPFs for any 
even integer weight.
We first state and prove
two lemmas
that will help 
us to specify the poles in the two RPF relations.

\begin{lemma}
\label{lemma:quadraticslash}
Let
\( \alpha \) be a hyperbolic fixed point of \( G_{p} \)
for \( p \geq 3 \),
and 
put \( Q_{\alpha}(z) = Q_{\alpha}(z,1) \).
Then 
for every \( M \in G_{p} \) we have
\begin{equation*}
\left( Q_{\alpha}^{-k} \mid M \right)(z)
	= Q_{M^{-1}\alpha}^{-k}(z).
\end{equation*}
\end{lemma}

\begin{remark}
Schmidt first stated this 
in \cite[p. 236]{MR1219337}
without proof.
We include a proof for completeness.
\end{remark}

\begin{proof}
Write 
\( Q_{\alpha}(z) = Az^{2}+Bz+C = A(z-\alpha)(z-\alpha^{\prime}) \)
and 
\( M = \begin{pmatrix}
a & b \\
c & d
\end{pmatrix} \).
Then
\( Q_{M^{-1}\alpha}(z) = \tilde{A}(z-M^{-1}\alpha)(z-M^{-1}\alpha^{\prime}) \),
where
\( \tilde{A} = Aa^{2} + Bac + Cc^{2} \).
We calculate that
\begin{align*}
\left( Q_{\alpha}^{-k} \mid M \right)(z)
& = (cz+d)^{-2k}A^{-k}
	\left( \frac{az+b}{cz+d} - \alpha \right)^{-k}
	\left( \frac{az+b}{cz+d} - \alpha^{\prime} \right)^{-k} \\
& = A^{-k}
	\left( (-c\alpha + a)z - (d\alpha - b) \right)^{-k}
	\left( (-c\alpha^{\prime} + a)z - (d\alpha^{\prime} - b) \right)^{-k} \\
& = A^{-k}(-c\alpha + a)^{-k}(-c\alpha^{\prime} + a)^{-k}
	\left( z - \frac{d\alpha - b}{-c\alpha + a} \right)^{-k}
	\left( z - \frac{d\alpha^{\prime} - b}{-c\alpha^{\prime} + a} \right)^{-k} \\
& = A^{-k}(a^{2} - (\alpha + \alpha^{\prime})ac + \alpha\alpha^{\prime}c^{2})^{-k}
	\left( z - M^{-1}\alpha \right)^{-k}
	\left( z - M^{-1}\alpha^{\prime} \right)^{-k} \\
& = (Aa^{2} + Bac + Cc^{2})^{-k}
	\left( z - M^{-1}\alpha \right)^{-k}
	\left( z - M^{-1}\alpha^{\prime} \right)^{-k} \\
& = Q_{M^{-1}\alpha}^{-k}(z).
\end{align*}
\end{proof}

The next lemma specifies
the poles that result
when we apply 
the second relation
to part of the RPF
in \eqref{eq:RPFform}.
We consider poles in the extended complex plane,
so we include the point at \( \infty \)
in our analysis.

\begin{lemma}
\label{lemma:q2slash}
Let \( U = U_{\lambda_{p}} \) for \( p \geq 3 \)
and let \( k \in \Z^{+} \).
Suppose that
\( r(z) = \sum_{n=1}^{2k-1} \frac{c_{n}}{z^{n}} \)
is a nontrivial function.
Then for \( 1 \leq t \leq p-2 \),
the function \( \left( r \mid U^{t} \right)(z) \)
has poles at \( z = U^{p-t}(0) \)
and \( z = U^{p-t+1}(0) \),
and no other poles
in 
\( \C \cup \{ \infty \} \).
The function \( \left( r \mid U^{p-1} \right)(z) \)
has a pole at \( z = U^{2}(0) \),
and no other poles
in \( \C \cup \{ \infty \} \).
\end{lemma}

\begin{proof}
If we let 
\( \gamma_{t} = \frac{\sin(t\pi/p)}{\sin(\pi/p)} \)
then 
\( U^{t} = \begin{pmatrix}
\gamma_{t+1} & -\gamma_{t} \\
\gamma_{t} & -\gamma_{t-1}
\end{pmatrix} \)
for \( t \in \Z \)
\cite{MR774398}.
For \( 1 \leq t \leq p-1 \)
we have
\begin{equation*}
\left( r \mid U^{t} \right)(z)
	= (\gamma_{t}z - \gamma_{t-1})^{-2k}q(U^{t}z).
\end{equation*}
The automorphy factor
\( (\gamma_{t}z - \gamma_{t-1})^{-2k} \)
has a pole of order \( 2k \)
at
\begin{equation*}
z = \frac{\gamma_{t-1}}{\gamma_{t}}
	= \frac{1}{U^{t}(0)}
	= U^{p-t+1}(0).
\end{equation*}
We have used Lemma 9
in \cite{MR2582982}
for the last equality.
The automorphy factor
has a zero at \( \infty \)
of order \( 2k \).
The function
\( r(U^{t}z) \)
has a pole at \( z = U^{p-t}(0) \)
of order less than \( 2k \),
and a zero at \( z = U^{p-t+1}(0) \)
of order less than \( 2k \),
and no other poles or zeros.
When we multiply expressions,
the zero at \( z=U^{p-t+1}(0) \) 
(from \( r(U^{t}z) \))
combines with the pole at \( z=U^{p-t+1}(0) \) 
(from \( (\gamma_{t}z - \gamma_{t-1})^{-2k} \))
but does not cancel it
because the order of the zero is less than \( 2k \).

If \( t = p-1 \),
we have the additional combination of
the pole at \( z = U(0) = \infty \)
(from \( r(U^{p-1}z) \))
with the zero at \( z = \infty \)
(from \( (\gamma_{p-1}z - \gamma_{p-2})^{-2k} \))
to give no pole at \( z = \infty \).
The poles that remain
are the ones given in the statement of the lemma.
\end{proof}

\begin{theorem}
Let
\begin{equation}
q(z) = \sum_{\alpha \in Z_{\mathcal{A}}} Q_{\alpha}(z,1)^{-k}
	- (-1)^{k}\sum_{\alpha \in Z_{-\mathcal{A}}} Q_{\alpha}(z,1)^{-k},
\label{eq:RPFquadraticanyk1}
\end{equation}
where each
\( \mathcal{A}_{\ell} \) 
is a \( G_{p} \)-equivalence class of 
\( \Z[\lambda_{p}] \)-BQFs.
Then \( q \)
is a rational period function 
of weight \( 2k \) on \( G_{p} \)
for any positive integer \( k \).
\end{theorem}

\begin{remark}
\mbox{}
\begin{enumerate}
\renewcommand{\labelenumi}{(\roman{enumi})}
\item 
This generalizes the 
class of RPFs
on the modular group given
in Theorem 3.2 of \cite{MR1045397},
which (with an appropriate change of notation)
is also given
in Theorem 3 of \cite{MR2734690}.
\item 
We can say more about the RPF in \eqref{eq:RPFquadraticanyk1}
for specific cases.
\begin{itemize}
\item 
If \( \mathcal{A} = -\mathcal{A} \) and \( k \) is odd
then 
\begin{equation*}
q(z) = 2\sum_{\alpha \in Z_{\mathcal{A}}} Q_{\alpha}(z,1)^{-k},
\end{equation*}
which is contained in
the expression in Theorem \ref{thm.SymmetricISPcorrected}.
This RPF has only one ISP and it is Hecke-symmetric.
\item 
If \( \mathcal{A} = -\mathcal{A} \) and \( k \) is even
then \( q(z) \equiv 0 \).
\item 
If \( \mathcal{A} \neq -\mathcal{A} \) (for any \( k \))
then 
\( q \) has two non-symmetric ISPs
\( P_{\mathcal{A}} \) and \( P_{-\mathcal{A}} \),
but the complete set of poles
\( P_{\mathcal{A}} \cup P_{-\mathcal{A}} \) 
is Hecke-symmetric.
\end{itemize}
\end{enumerate}
\end{remark}

\begin{proof}
We choose \( L = 2 \) in \eqref{eq:RPFform},
put \( \mathcal{A}_{1} = \mathcal{A} \)
and \( \mathcal{A}_{2} = -\mathcal{A} \),
and note that 
\( \mathcal{A}_{1} \) and \( \mathcal{A}_{2} \)
have the same discriminant \( D \).
If we put
\( C_{1} = D^{-k/2} \),
\( C_{2} = (-1)^{k+1}D^{-k/2} \),
and \( c_{0} = 0 \)
the RPF
given by \eqref{eq:RPFform}
is
\begin{align}
q(z)
  & = D^{-k/2}
	  \left( \sum_{\alpha \in Z_{\mathcal{A}}} q_{k,\alpha}(z) 
      - \sum_{\alpha \in Z_{-\mathcal{A}}} q_{k,\alpha^{\prime}}(z) \right) \nonumber \\
  & \hspace{.5cm} + (-1)^{k+1}D^{-k/2}
	  \left( \sum_{\alpha \in Z_{-\mathcal{A}}} q_{k,\alpha}(z) 
      - \sum_{\alpha \in Z_{\mathcal{A}}} q_{k,\alpha^{\prime}}(z) \right)
      + \sum_{n=1}^{2k-1}\frac{c_{n}}{z^{n}} \nonumber \\
  & = D^{-k/2}\sum_{\alpha \in Z_{\mathcal{A}}} 
  \left( q_{k,\alpha}(z) + (-1)^{k} q_{k,\alpha^{\prime}}(z) \right) \nonumber \\
  & \hspace{.5cm} - D^{-k/2}\sum_{\alpha \in Z_{-\mathcal{A}}} 
  \left( (-1)^{k} q_{k,\alpha}(z) + q_{k,\alpha^{\prime}}(z) \right)
    + \sum_{n=1}^{2k-1}\frac{c_{n}}{z^{n}}.
\label{eq:RPFquadraticanyk2}
\end{align}
The functions 
\( q_{k,\alpha}(z) \) and \( q_{k,\alpha^{\prime}}(z) \)
are given by 
\eqref{eq:PPatalpha},
and for \( q_{k,\alpha^{\prime}}(z) \)
we calculate that
\begin{align*}
q_{k,\alpha^{\prime}}(z)
& = PP_{\alpha^{\prime}} 
         \left[ \frac{(\alpha^{\prime}-\alpha)^{k}}
	         {(z-\alpha^{\prime})^{k}(z-\alpha)^{k}}
                 \right] 
 = (-1)^{k}PP_{\alpha^{\prime}} 
         \left[ \frac{(\alpha-\alpha^{\prime})^{k}}{(z-\alpha)^{k}(z-\alpha^{\prime})^{k}}
                 \right],
\end{align*}
so
\begin{equation*}
q_{k,\alpha}(z) + (-1)^{k}q_{k,\alpha^{\prime}}(z)
	= \frac{D^{k/2}}{Q_{\alpha}(z,1)^{k}}.
\end{equation*}
With this \eqref{eq:RPFquadraticanyk2} becomes
\begin{equation*}
q(z) = \sum_{\alpha \in Z_{\mathcal{A}}} Q_{\alpha}(z,1)^{-k}
	- (-1)^{k}\sum_{\alpha \in Z_{-\mathcal{A}}} Q_{\alpha}(z,1)^{-k}
	+ \sum_{n=1}^{2k-1}\frac{c_{n}}{z^{n}}.
\label{eq:RPFquadraticanyk3}
\end{equation*}
It remains to show that 
\( \sum_{n=1}^{2k-1}\frac{c_{n}}{z^{n}} = 0 \).
To that end we put
\begin{equation*}
q_{1}(z) 
	= \sum_{\alpha \in Z_{\mathcal{A}}} Q_{\alpha}(z,1)^{-k}
	- (-1)^{k}\sum_{\alpha \in Z_{-\mathcal{A}}} Q_{\alpha}(z,1)^{-k},
\end{equation*}
and 
\begin{equation*}
q_{2}(z) 
	= \sum_{n=1}^{2k-1}\frac{c_{n}}{z^{n}},
\end{equation*}
so that
\( q(z) = q_{1}(z) + q_{2}(z) \).
We suppose by way of contradiction
that \( q_{2} \)
is nontrivial.
We will show that 
this implies that \( q_{2} \) is itself an RPF,
which contradicts
the fact that 
\( q_{2} \)
must have the form \eqref{eq:RPFPoleatZero}.

Now
\( q_{1} \) and \( q_{1} \mid T \)
do not have poles at \( z=0 \),
by Lemma \ref{lemma:quadraticslash}.
On the other hand,
the only pole
in \( \C \cup \{ \infty \} \)
for
\( q_{2} \)
and
\( q_{2}\mid T \)
is at \( z=0 \).
We use these observations along with
the fact that \( q \) satisfies the first relation \eqref{eq:first_relation}
to calculate
\begin{align*}
0
& = PP_{0}\left[ q + q \mid T \right] \\
& = PP_{0}\left[ q_{2} + q_{2} \mid T \right] \\
& = q_{2} + q_{2} \mid T,
\end{align*}
so \( q_{2} \) also satisfies the first relation \eqref{eq:first_relation}.

Suppose
\( 1 \leq t \leq p-1 \).
Now \( q_{1} \)
and 
\( q_{1}\mid U^{t} \)
(by Lemma \ref{lemma:quadraticslash})
have poles only at hyperbolic points.
On the other hand,
the only poles 
in \( \C \cup \{ \infty \} \)
for
\( q_{2} \)
and 
\( q_{2}\mid U^{t} \)
(by Lemma \ref{lemma:q2slash})
are at one or two of the parabolic points
\( \left\{ U^{s}(0): 2 \leq s \leq p \right\}  \).
We use these observations along with
the fact that \( q \) satisfies the second relation \eqref{eq:second_relation}
to calculate
\begin{align*}
0
& = \sum_{2 \leq s \leq p} PP_{U^{s}(0)}
	\left[ q + q \mid U + \cdots q \mid U^{p-1} \right] \\
& = \sum_{2 \leq s \leq p} PP_{U^{s}(0)}
	\left[ q_{2} + q_{2} \mid U + \cdots q_{2} \mid U^{p-1} \right] \\
& = q_{2} + q_{2} \mid U + \cdots q_{2} \mid U^{p-1},
\end{align*}
so \( q_{2} \) also satisfies the second relation \eqref{eq:second_relation}.
But then
\( q_{2} \) must be an RPF,
which is a contradiction.
Thus \( q_{2}(z)\equiv 0 \),
and  \( q \) has the form 
given in the statement of the theorem.
\end{proof}

\section{An RPF with a non-symmetric pole set}
\label{sec:RPFs_non-symmetric}

If \( \mathcal{A} \)
is an equivalence class of 
\( \lambda \)-BQFs
with \( \mathcal{A} \neq -\mathcal{A} \),
we can 
use \eqref{eq:RPFform}
to write a rational period function
that has a single nonzero
irreducible system of poles 
\( P_{\mathcal{A}} 
	= \mathcal{Z}_{\mathcal{A}} \cup \mathcal{Z}_{-\mathcal{A}^{\prime}} \),
which does not have Hecke-symmetry.
We illustrate this with an example.

\begin{example}
Put \( \lambda = \lambda_{4} = \sqrt{2} \)
and consider the \( \lambda_{4} \)-BQFs
of discriminant \( D=14 \).
We let \( \mathcal{A} \)
be the equivalence class \( \mathcal{A}_{1} \)
from Example \ref{ex:RPF_nonsym_ISP},
so that (using the notation from Example \ref{ex:RPF_nonsym_ISP})
\begin{align*}
P_{\mathcal{A}}
& = \mathcal{Z}_{\mathcal{A}} \cup \mathcal{Z}_{-\mathcal{A}^{\prime}} \\
& =  \left\{ \alpha_{1}, \alpha_{2}, \beta_{1}^{\prime}, 
		\beta_{2}^{\prime} \right\}, \\
& =  \left\{ \frac{\sqrt{2}+\sqrt{14}}{2}, \frac{-\sqrt{2}+\sqrt{14}}{2},
		\frac{\sqrt{2}-\sqrt{14}}{6}, \frac{-\sqrt{2}-\sqrt{14}}{6} \right\}.
\end{align*}
We put \( L=1 \),
\( C_{1} = 0 \) and \( c_{0} = 0 \)
in \eqref{eq:RPFform}
and have the RPF on \( G_{4} \)
\begin{align*}
q(z)
& = \sum_{\alpha \in Z_{\mathcal{A}}} q_{k,\alpha}(z) 
		- \sum_{\alpha \in Z_{-\mathcal{A}}} q_{k,\alpha^{\prime}}(z)
	    + \sum_{n=1}^{2k-1}\frac{c_{n}}{z^{n}} \\
& = q_{k,\alpha_{1}}(z) + q_{k,\alpha_{2}}(z) 
		- q_{k,\beta_{1}^{\prime}}(z) - q_{k,\beta_{2}^{\prime}}(z)
	    + \sum_{n=1}^{2k-1}\frac{c_{n}}{z^{n}}.
\end{align*}
We let \( k=1 \)
and calculate that
\begin{align*}
q(z)
& = \frac{1}{z-\alpha_{1}} + \frac{1}{z-\alpha_{2}}
	- \frac{1}{z-\beta_{1}^{\prime}} - \frac{1}{z-\beta_{2}^{\prime}} 
	+ \frac{c_{1}}{z}.
\end{align*}
We have from \eqref{eq:RPFPoleatZero}
that
\( \frac{c_{1}}{z} \)
is itself an RPF of weight \( 2 \) on \( G_{4} \)
so we may let \( c_{1} = 0 \), and 
\begin{align*}
q(z)
& = \frac{1}{z-\alpha_{1}} + \frac{1}{z-\alpha_{2}}
	- \frac{1}{z-\beta_{1}^{\prime}} - \frac{1}{z-\beta_{2}^{\prime}}  \\
& = \left( z - \frac{\sqrt{2}+\sqrt{14}}{2} \right)^{-1} 
	+ \left( z - \frac{-\sqrt{2}+\sqrt{14}}{2} \right)^{-1} \\
&\hspace{2cm}	- \left( z - \frac{\sqrt{2}-\sqrt{14}}{6} \right)^{-1}  
	+ \left( z - \frac{-\sqrt{2}-\sqrt{14}}{6} \right)^{-1}
\end{align*}
is an RPF of weight \( 2 \) on \( G_{4} \)
with 
one ISP
which does not have Hecke-symmetry.
\end{example}

\bibliographystyle{alpha}    

\bibliography{WRLibraryMR,WRLibraryNotMR}

\begin{thebibliography}{DIT10}

\bibitem[Ash89]{MR980298}
Avner Ash.
\newblock Parabolic cohomology of arithmetic subgroups of {${\rm SL}(2,{\bf
  Z})$} with coefficients in the field of rational functions on the {R}iemann
  sphere.
\newblock {\em Amer. J. Math.}, 111(1):35--51, 1989.

\bibitem[CP90]{MR1045397}
YoungJu Choie and L.~Alayne Parson.
\newblock Rational period functions and indefinite binary quadratic forms. {I}.
\newblock {\em Math. Ann.}, 286(4):697--707, 1990.

\bibitem[CP91]{MR1103673}
YoungJu Choie and L.~Alayne Parson.
\newblock Rational period functions and indefinite binary quadratic forms.
  {II}.
\newblock {\em Illinois J. Math.}, 35(3):374--400, 1991.

\bibitem[CR01]{MR1876701}
Wendell Culp-Ressler.
\newblock Rational period functions on the {H}ecke groups.
\newblock {\em Ramanujan J.}, 5(3):281--294, 2001.

\bibitem[CZ93]{MR1210514}
Yj. Choie and D.~Zagier.
\newblock Rational period functions for {${\rm PSL}(2,\bold Z)$}.
\newblock In {\em A tribute to {E}mil {G}rosswald: number theory and related
  analysis}, volume 143 of {\em Contemp. Math.}, pages 89--108. Amer. Math.
  Soc., Providence, RI, 1993.

\bibitem[DIT10]{MR2734690}
W.~Duke, {\"O}.~Imamo{\=g}lu, and {\'A}.~T{\'o}th.
\newblock Rational period functions and cycle integrals.
\newblock {\em Abh. Math. Semin. Univ. Hambg.}, 80(2):255--264, 2010.

\bibitem[HR13]{MR3078226}
Giabao Hoang and Wendell Ressler.
\newblock Conjugacy classes and binary quadratic forms for the {H}ecke groups.
\newblock {\em Canad. Math. Bull.}, 56(3):570--583, 2013.

\bibitem[Kno74]{MR0344454}
Marvin~I. Knopp.
\newblock Some new results on the {E}ichler cohomology of automorphic forms.
\newblock {\em Bull. Amer. Math. Soc.}, 80:607--632, 1974.

\bibitem[Kno78]{MR0485700}
Marvin~I. Knopp.
\newblock Rational period functions of the modular group.
\newblock {\em Duke Math. J.}, 45(1):47--62, 1978.
\newblock With an appendix by Georges Grinstein.

\bibitem[Kno81]{MR623004}
Marvin~I. Knopp.
\newblock Rational period functions of the modular group. {II}.
\newblock {\em Glasgow Math. J.}, 22(2):185--197, 1981.

\bibitem[MR84]{MR774398}
Holger Meier and Gerhard Rosenberger.
\newblock Hecke-{I}ntegrale mit rationalen periodischen {F}unktionen und
  {D}irichlet-{R}eihen mit {F}unktionalgleichung.
\newblock {\em Results Math.}, 7(2):209--233, 1984.

\bibitem[Par93]{MR1210515}
L.~Alayne Parson.
\newblock Rational period functions and indefinite binary quadratic forms.
  {III}.
\newblock In {\em A tribute to {E}mil {G}rosswald: number theory and related
  analysis}, volume 143 of {\em Contemp. Math.}, pages 109--116. Amer. Math.
  Soc., Providence, RI, 1993.

\bibitem[Res09]{MR2582982}
Wendell Ressler.
\newblock On binary quadratic forms and the {H}ecke groups.
\newblock {\em Int. J. Number Theory}, 5(8):1401--1418, 2009.

\bibitem[Ros54]{MR0065632}
David Rosen.
\newblock A class of continued fractions associated with certain properly
  discontinuous groups.
\newblock {\em Duke Math. J.}, 21:549--563, 1954.

\bibitem[Sch93]{MR1219337}
Thomas~A. Schmidt.
\newblock Remarks on the {R}osen {$\lambda$}-continued fractions.
\newblock In {\em Number theory with an emphasis on the {M}arkoff spectrum
  ({P}rovo, {UT}, 1991)}, volume 147 of {\em Lecture Notes in Pure and Appl.
  Math.}, pages 227--238. Dekker, New York, 1993.

\bibitem[Sch96]{MR1378572}
Thomas~A. Schmidt.
\newblock Rational period functions and parabolic cohomology.
\newblock {\em J. Number Theory}, 57(1):50--65, 1996.

\bibitem[SS95]{MR1362251}
Thomas~A. Schmidt and Mark Sheingorn.
\newblock Length spectra of the {H}ecke triangle groups.
\newblock {\em Math. Z.}, 220(3):369--397, 1995.

\end{thebibliography}

\end{document}